\documentclass{amsart}

\usepackage[normalem]{ulem}

\usepackage{lineno}

\usepackage{soul}

\usepackage{amssymb, amsmath, amsthm, hyperref, euscript, enumerate, color}
\usepackage[dvipsnames]{xcolor}

\usepackage{amscd}

\usepackage{bbm}

\usepackage{enumitem}

\usepackage[english]{babel}

\numberwithin{equation}{section}

\theoremstyle{plain}

\newtheorem{theorem}{Theorem}

\newtheorem{corollary}{Corollary}

\newtheorem{proposition}{Proposition}

\newtheorem{remark}{Remark}

\newtheorem{lemma}{Lemma}

\newtheorem{exa}{Example}

\newtheorem{thm}{Theorem}

\newtheorem{cor}[thm]{Corollary}

\theoremstyle{definition}

\newcommand{\Q}{\mathbb{Q}}

\newcommand{\R}{\mathbb{R}}

\newcommand{\Z}{\mathbb{Z}}

\newcommand{\C}{\mathbb{C}}

\newcommand{\id}{\textnormal{id}}

\DeclareMathOperator{\Sw}{SW}

\DeclareMathOperator{\Sp}{Spin}

\DeclareMathOperator{\Ks}{KS}

\makeatletter
\@namedef{subjclassname@2020}{%
  \textup{2020} Mathematics Subject Classification}
\makeatother

\begin{document}

\author{Rafael Torres}

\title[Knotted and TOP isotopic 2-spheres with cyclic 2-knot group]{Topologically isotopic and smoothly inequivalent 2-spheres in simply connected 4-manifolds whose complement has a prescribed  fundamental group.}

\address{Scuola Internazionale Superiori di Studi Avanzati (SISSA)\\ Via Bonomea 265\\34136\\Trieste\\Italy}

\email{rtorres@sissa.it}


\maketitle

\emph{Abstract}: We describe a procedure to construct infinite sets of pairwise smoothly inequivalent 2-spheres in simply connected 4-manifolds, which are topologically isotopic and whose complement has a prescribed fundamental group that satisfies some conditions. This class of groups include cyclic groups and the binary icosahedral group. These are the first known examples of such exotic embeddings of 2-spheres in 4-manifolds. Examples of locally flat embedded 2-spheres in a non-smoothable 4-manifold whose complements are homotopy equivalent to smoothly embedded ones are also given.

\section{Main results}\label{Introduction}

The first main result of this note is the following theorem.

\begin{thm}\label{Theorem A}Fix $p\geq 2$. There is an infinite set\begin{center}$\{S_{n, p}: n\in \Z\}$\end{center}of smoothly embedded 2-spheres in $2\mathbb{CP}^2\#4\overline{\mathbb{CP}^2}$ that satisfies the following properties.
\begin{itemize}
\item Any two elements are topologically isotopic.
\item There is a diffeomorphism of pairs\begin{center}$(2\mathbb{CP}^2\#4\overline{\mathbb{CP}^2}, S_{n_1, p})\rightarrow (2\mathbb{CP}^2\#4\overline{\mathbb{CP}^2}, S_{n_2, p})$\end{center} if and only if $n_1 = n_2$.
\item The fundamental group of the complement is\begin{center}$\pi_1(2\mathbb{CP}^2\#4\overline{\mathbb{CP}^2}\setminus \nu(S_{n, p})) = \Z/p$\end{center} for every $n\in \Z$.
\item $[S_{n, p}]\neq 0\in H_2(2\mathbb{CP}^2\#4\overline{\mathbb{CP}^2}; \Z)$ for every $n\in \Z$.
\item Surgery along each of these 2-spheres yields an infinite set of pairwise homeomorphic and pairwise non-diffeomorphic closed smooth 4-manifolds with fundamental group $\Z/p$.
\end{itemize}
\end{thm}

Theorem \ref{Theorem A} provides the first known example of an infinite set of 2-spheres smoothly embedded in a simply connected 4-manifold that are pairwise topologically isotopic, pairwise smoothly inequivalent and having a complement with finite cyclic fundamental group. Schwartz \cite[Theorem 2]{[Schwartz]} pointed out the existence of closed simply connected 4-manifolds containing pairs of smoothly embedded 2-spheres that are both smoothly equivalent and topologically isotopic, but not smoothly isotopic. Examples of these exotic embeddings of 2-spheres in closed 4-manifolds with simply connected complement have been constructed by Akbulut \cite{[Akbulut2]} and Auckly-Kim-Melvin-Ruberman \cite{[AuckleyKimMelvinRuberman]}. Exotic embeddings of surfaces with positive genus in simply connected 4-manifolds and complement having non-trivial fundamental group are found in Kim \cite{[Kim]} and Kim-Ruberman \cite{[KimRuberman]}. An ingredient in the proof of Theorem \ref{Theorem A} is of independent interest: we point out in Theorem \ref{Theorem Infinite Smooth Structures} that constructions of inequivalent smooth structures on simply connected 4-manifolds of Fintushel-Stern \cite{[FintushelStern2], [FintushelStern]} can be extended to produce such structures on 4-manifolds with non-trivial fundamental group too. 

The second main result provides a construction procedure of topologically equivalent yet smoothly inequivalent homologically essential 2-spheres whose complement can be chosen to have the same fundamental group as a wide range of $\Q$-homology 4-spheres. We work with the modified Seiberg-Witten $\Sw'_X$ invariant of a closed 4-manifold $X$ as defined, for example, in \cite[Section 2]{[FintushelParkStern]}, and denote by $\mathcal{B}_X$ the set of basic classes.


\begin{thm}\label{Theorem B}Let $\{Z_n: n\in \Z\}$ be an infinite set of closed smooth simply connected 4-manifolds with pairwise different integer invariants\begin{equation}\label{Integer Invariant}S_n = \max\{|\Sw'_{Z_n}(k_{Z_n})| :k_{Z_n}\in \mathcal{B}_{Z_n}\},\end{equation} which are pairwise homeomorphic to a given closed 4-manifold $Z$ and such that the connected sum $Z_n\#S^2\times S^2$ is diffeomorphic to $Z\# S^2\times S^2$ for every $n\in \Z$. Let $M$ be a closed smooth 4-manifold with $H_\ast(M; \Q) \cong H_\ast(S^4; \Q)$ and set $\pi:=\pi_1M$. Suppose that there is a loop $\alpha\subset M$ and a choice of framing such that\begin{equation}\label{Stabilization}S^2\times S^2 = M\setminus \nu(\alpha)\cup D^2\times S^2.\end{equation}

There is an infinite set\begin{center}$\{S_{n, \pi}: n\in \Z\}$\end{center}of smoothly embedded 2-spheres in $Z\#S^2\times S^2$ that satisfies the following properties.
\begin{itemize}
\item There is a homeomorphism of pairs\begin{center}$(Z\#S^2\times S^2, S_{n_1, \pi})\rightarrow (Z\#S^2\times S^2, S_{n_2, \pi})$\end{center} for every $n_i\in \Z$.
\item There is a diffeomorphism of pairs\begin{center}$(Z\#S^2\times S^2, S_{n_1, \pi})\rightarrow (Z\#S^2\times S^2, S_{n_2, \pi})$\end{center} if and only if $n_1 = n_2$.
\item The fundamental group of the complement is\begin{center}$\pi_1(Z\#S^2\times S^2\setminus \nu(S_{n, \pi})) = \pi$\end{center} and its homology class satisfies\begin{center}$[S_{n, \pi}]\neq 0\in H_2(Z\# S^2\times S^2; \Z)$\end{center} for every $n\in \Z$.
\item Surgery along each of these 2-spheres yields an infinite set $\{Z_n\# M: n\in \Z\}$ of pairwise non-diffeomorphic closed smooth 4-manifolds with fundamental group $\pi$, and that are pairwise homeomorphic to the connected sum $Z\#M$.
\end{itemize}
\end{thm}

Please see \cite[Proof of Theorem 1]{[FintushelParkStern]} for details on the definition of the invariant (\ref{Integer Invariant}). Fintushel-Stern constructed infinite sets as in the hypothesis of Theorem \ref{Theorem B} for $Z = \mathbb{CP}^2\#k\overline{\mathbb{CP}^2}$ for $2\leq k \leq 7$ in \cite{[FintushelStern2], [FintushelStern]}. Baykur-Sunukian \cite{[BaykurSunukjian]} showed that Fintushel-Stern's examples become diffeomorphic after a connected sum with a single copy of $S^2\times S^2$. Examples of $\Q$-homology 4-spheres $M$ that satisfy the hypothesis are spun 4-manifolds with the fundamental group of any lens space and the Poincar\'e homology 3-sphere. A similar result holds if (\ref{Stabilization}) is substituted for the non-trivial bundle $S^2\widetilde{\times} S^2$. It is possible to strengthen the conclusion of Theorem \ref{Theorem B} to topologically isotopic 2-spheres, although we do not pursue this endeavor here; see Sunukjian \cite{[Sunukjian]}. 

A contribution of this note is to point out the simplicity of the proofs of Theorem \ref{Theorem A} and Theorem \ref{Theorem B}. The reader will notice that the 4-manifolds in the last clause of Theorem \ref{Theorem B} are smoothly reducible (cf. \cite[Definition 10.1.17]{[GompfStipsicz]}), while those in the last clause of Theorem \ref{Theorem A} are not. We explain in Remark \ref{Remark Explanation} how an instance of Theorem \ref{Theorem B} implies the claims on the existence of the homeomorphism of pairs and the non-existence of the diffeomorphism of pairs of Theorem \ref{Theorem A}.  An independent proof of Theorem \ref{Theorem A} is given in Section \ref{Section Proof of Theorem A} as well. The following consequence of Theorem \ref{Theorem B} is another contribution. 

\begin{cor}\label{Corollary C}Let $G$ be a finite cyclic group or the icosahedral group\begin{center}$G = \langle g_1, g_2 : {g_1}^5 = (g_1g_2)^2 = {g_2}^3\rangle$.\end{center}There is an infinite set of smoothly embedded 2-spheres in $2\mathbb{CP}^2\#4\overline{\mathbb{CP}^2}$ that are pairwise topologically equivalent, yet pairwise smoothly inequivalent, and the fundamental group of the complement is $G$.
\end{cor}

These are the first examples of exotic embeddings of 2-spheres in simply connected 4-manifolds whose complement has a fundamental group isomorphic to the binary icosahedral group among several other choices of groups. We exhibit interesting smooth embeddings of nullhomotopic 2-spheres in the fourth main result of this note.

\begin{thm}\label{Theorem E}There is an infinite set\begin{equation}\label{Set Inf}\{S_{n}: n\in \Z\}\end{equation}of 2-spheres smoothly embedded in $2\mathbb{CP}^2\#4\overline{\mathbb{CP}^2}$ that satisfies the following properties.
\begin{itemize}
\item The fundamental group of the complement of an element in (\ref{Set Inf}) is\begin{center}$\pi_1(2\mathbb{CP}^2\#4\overline{\mathbb{CP}^2}\setminus \nu(S_{n})) = \Z$\end{center} and $[S_{n}]= 0\in \pi_2(2\mathbb{CP}^2\#4\overline{\mathbb{CP}^2})$ for every $n\in \Z$.
\item There is a diffeomorphism of pairs\begin{center}$(2\mathbb{CP}^2\#4\overline{\mathbb{CP}^2}, S_{n_1})\rightarrow (2\mathbb{CP}^2\#4\overline{\mathbb{CP}^2}, S_{n_2})$\end{center} if and only if $n_1 = n_2$.
\end{itemize}
\end{thm}

Notice that elements in (\ref{Set Inf}) do not bound a smoothly embedded 3-ball in $2\mathbb{CP}^2\#4\overline{\mathbb{CP}^2}$. The smoothly inequivalent embeddings of homotopically trivial 2-spheres of Theorem \ref{Theorem E} are related to a construction of an infinite set of closed smooth 4-manifolds with infinite cyclic fundamental group and the homology of the connected sum $2\mathbb{CP}^2\#4\overline{\mathbb{CP}^2}\#S^1\times S^3$, which is given in Theorem \ref{Theorem Infinite Smooth Structures Infinite}.

While any 2-sphere in a closed simply connected 4-manifold can be assumed to be regularly immersed,  Hambleton-Kreck \cite{[HambletonKreck]} and Lee-Wilczy\'nski \cite{[LeeWilczynski], [LeeWilczynski2]} completely characterized when a homology class of non-zero divisibility can be represented by a locally flat embedded 2-sphere. The fifth and last result to be mentioned in this introduction records the existence of a myriad of explicit examples of locally flat embedded 2-spheres in closed simply connected 4-manifolds whose exteriors are homotopy equivalent but not homeomorphic.

\begin{thm}\label{Theorem D}For every $p\geq2$, there is a locally flat embedded 2-sphere\begin{equation}\label{Locally Flat Sphere}S_p\subset \ast \mathbb{CP}^2\#\overline{\mathbb{CP}^2}\end{equation} whose complement has finite cyclic group $\Z/p$, and it is homotopy equivalent to the complement of a smoothly embedded 2-sphere\begin{equation}\label{Smoothly Embedded Sphere}S_p'\subset\mathbb{CP}^2\#\overline{\mathbb{CP}^2}.\end{equation} 
\end{thm}

Theorem \ref{Theorem D} is essentially derived from an existence result of non-smoothable $\Q$-homology 4-spheres due to Hambleton-Kreck \cite{[HambletonKreck2]}. Other interesting examples are found in Kasprowski-Lambert-Cole-Land-Lecuona  \cite{[KasprowskiLambertColeLandLecuona]}.

\subsection{Acknowledgements:}E-mail exchanges with Aru Ray, Bob Gompf, Danny Ruberman and Ian Hambleton have increased our understanding on the subject. We are grateful to them for patiently sharing their knowledge. We thank the anonymous referees for their careful reading of the manuscript and for their input, which helped us improve the manuscript. Del tega dela je nastal na Rogu in Metelkovi.

\section{Proofs}\label{Section Proofs} 

\subsection{Infinitely many inequivalent smooth structures}\label{Section Infinitely Many} Fintushel-Stern \cite[Theorem 1]{[FintushelStern]} showed that there is a nullhomologous 2-torus $T$ smoothly embedded in $\mathbb{CP}^2\#3\overline{\mathbb{CP}^2}$ such that performing surgeries on $T$ results in infinitely many inequivalent smooth structures on $\mathbb{CP}^2\# 3\overline{\mathbb{CP}^2}$. We point out that changing the coefficients of the torus surgery on $T$ introduces homotopically non-trivial loops to the resulting 4-manifold, and their procedure also yields infinitely many smooth structures on 4-manifolds with prescribed cyclic fundamental group. The latter will serve as raw material to construct the knotted 2-spheres. 

We introduce terminology to state the result and follow the notation in \cite[Section 3]{[FintushelStern]}. Let $T\subset X$ be a smoothly embedded 2-torus with trivial tubular neighborhood $\nu(T) = T^2\times D^2$. Let $\{a, b\}$ be loops in $T$ that form a symplectic basis of $\pi_1T = \Z^2$, and let $\{S^1_a, S^1_b\}$ be loops in $\partial\nu(T) = T^2\times \partial D^2 = T^2\times S^1$ that are homologous to $a$ and $b$, respectively.  The meridian of $T$ is denoted by $\mu_T$ and it is any curve in the same isotopy class of the curve $\{x\}\times \partial D^2\subset \partial \nu(T)$. The smooth 4-manifold\begin{equation}\label{4-manifold torus surgery}X_{T, S^1_b}(p/n):= (X\setminus \nu(T))\cup_{\varphi} (T^2\times D^2)\end{equation}where the gluing diffeomorphism satisfies $\varphi_\ast([\partial D^2]) = n[S^1_b] + p[\mu_T]$ is said to be obtained by performing a $p/n$-torus surgery to $X$ on $T$ along the curve $b$.

We first consider the case of finite cyclic fundamental group and postpone the infinite cyclic case to the end of the section. 

\begin{theorem}\label{Theorem Infinite Smooth Structures}Fix $p\geq 2$. There is a smoothly embedded nullhomologous 2-torus $T\subset \mathbb{CP}^2\#3\overline{\mathbb{CP}^2}$ and a nullhomologous curve in its complement $S^1_b\subset \mathbb{CP}^2\#3\overline{\mathbb{CP}^2}\setminus \nu(T)$ such that performing a $(p/n)$-torus surgery to $\mathbb{CP}^2\#3\overline{\mathbb{CP}^2}$ on $T$ along $S^1_b$ yields an infinite set\begin{equation}\label{Infinite Set}\{X_{T, S^1_b}(p/n): n\in \Z\}\end{equation}of pairwise non-diffeomorphic 4-manifolds such that every element is homeomorphic to the connected sum\begin{equation}\label{Topological Prototype}\mathbb{CP}^2\#3\overline{\mathbb{CP}^2}\# \Sigma_p,\end{equation}where $\Sigma_p$ is a $\Q$-homology 4-sphere with fundamental group $\pi_1\Sigma_p = \Z/p$.
\end{theorem}

\begin{proof} The only contribution in this note to the work of Fintushel-Stern \cite{[FintushelStern2], [FintushelStern]} that provides a proof of Theorem \ref{Theorem Infinite Smooth Structures} is the change in a coefficient of the torus surgery. We then employ a homeomorphism criteria of Hambleton-Kreck to pin down the homeomorphism class of the closed 4-manifolds that are constructed this way. Set $X:= \mathbb{CP}^2\# 3\overline{\mathbb{CP}^2}$ so to not overload the notation. Fintushel-Stern showed the existence of a nullhomologous torus $T\subset X$ \cite[Theorem 1.1]{[FintushelStern]} and the curve $b\subset T$ with framing $S^1_b\subset X\setminus \nu(T)$ as in the statement of Theorem \ref{Theorem Infinite Smooth Structures}. Build $X_{T, S^1}(p/n)$ as in (\ref{4-manifold torus surgery}). Since $[T] = 0\in H_2(X; \Z)$, we have that $H_1(X_{T, S^1_b}(p/n); \Z) = \Z/p$ for every $n\in \Z$; in this notation, $p = 0$ corresponds to $\Z$. We now fix $p\geq 2$.

To see that the fundamental group of $X_{T, S^1_b}(p/n)$ is $\Z/p$, we take a closer look at the constructions of Fintushel-Stern in \cite{[FintushelStern2], [FintushelStern]}, where six torus surgeries along six nullhomologous 2-tori $\{T_{1, i}, T_{2, i}: i = 1, 2 ,3\}$ are performed to $X$ to produce a symplectic 4-manifold $Q$ with fundamental group $\pi_1(Q) = \Z^6$ and that contains six Lagrangian 2-tori $\{L_{1, i}, L_{2, i}: i = 1, 2, 3\}$ \cite[Proposition 7]{[FintushelStern2]}, \cite[p. 77]{[FintushelStern]}. The complements of these 2-tori are the same, i.e.,\begin{equation}X\setminus \overset{3}{\underset{i = 1}\bigcup}\big(\nu(T_{1, i}) \sqcup \nu(T_{2, i})\big) = Q\setminus \overset{3}{\underset{i = 1}\bigcup}\big(\nu(L_{1, i}) \sqcup \nu(L_{2, i})\big).\end{equation}By applying six surgeries to the symplectic 4-manifold $Q$ along the Lagrangian 2-tori with a given choice of surgery curves \cite[Theorem 2]{[FintushelStern2]}, one obtains an infinite set of inequivalent smooth structures on $X$. The first five surgeries are $|1/1|$-torus surgeries, while the last one is a $1/n$-torus surgery \cite[p. 1685]{[FintushelStern2]}. In particular, this infinite set can be obtained by performing torus surgeries to $X$ on six nullhomologous 2-tori. For our purposes, we perform the first five surgeries verbatim as in the proof of \cite[Theorem 2]{[FintushelStern2]}, change the surgery coefficients of the sixth surgery to perform a $p/n$-torus surgery in order to obtain an infinite set\begin{equation}\label{Infinite Set New}\{X_{T, S^1_b}(p/n): n\in \Z\}\end{equation} for a fixed $p\geq 2$. It follows from the Seifert-van Kampen theorem that the fundamental group is $\pi_1(X_{T, S^1_b}(p/n)) = \Z/p$ \cite[p. 321]{[BaldridgeKirk]} for every $n\in \Z$: a detailed account on the computation of the fundamental group of the 4-manifolds obtained with such a change in the surgery coefficient can be found in several places in the literature, for example, \cite[Section 5]{[BaldridgeKirk]}, \cite[p. 595]{[AkhmedovPark]}. We have explained so far that six surgeries on six nullhomologous 2-tori in $X$ as in \cite[Theorem 2]{[FintushelStern2]} produce an infinite set (\ref{Infinite Set New}) of 4-manifolds with fundamental group $\Z/p$. 

We now appeal to the main result in \cite[Section 8]{[FintushelStern]}. Fintushel-Stern showed that the first five surgeries on $X$ do not change the diffeomorphism type of $X$ and, thus, that there is a single nullhomologous 2-torus $T\subset X$ along with a nullhomologous curve $S^1_b\subset X\setminus \nu(T)$ such that an $1/n$-torus surgery produce an infinite set of smooth structures on $X$, as we had mentioned before \cite[Theorem 1.1]{[FintushelStern]}. Thus, we conclude that each element in the set (\ref{Infinite Set New}) is obtained by performing a $p/n$-torus surgery on $T\subset X$ along $S^1_b$. 

We now argue that these 4-manifolds are homeomorphic to (\ref{Topological Prototype}). An inclusion-exclusion argument indicates that the Euler characteristic is unchanged under torus surgeries, i.e.,\begin{equation}\chi(X_{T, S^1_b}(p/n)) = \chi(X) = 6.\end{equation} Novikov additivity \cite[Remark 9.1.7]{[GompfStipsicz]} implies\begin{equation}\sigma(X_{T, S^1_b}(p/n)) = \sigma(X) = -2,\end{equation} and we conclude that the second Stiefel-Whitney class of $X_{T, S^1_b}(p/n)$ does not vanish employing a result of Rohklin cf. \cite[Theorem 1.2.29]{[GompfStipsicz]}. A classification result of Hambleton-Kreck \cite[Theorem C]{[HambletonKreck2]} allows us to conclude that the 4-manifold $X_{T, S^1_b}(p/n)$ is homeomorphic to $\mathbb{CP}^2\# 3\overline{\mathbb{CP}^2} \# \Sigma_p$, where $\Sigma_p$ is a closed smooth 4-manifold with Euler characteristic two and signature zero for every $n\in \Z$ and $p\geq 2$. 

To argue that we have constructed infinitely many 4-manifolds that are pairwise non-diffeomorphic, we compute their Seiberg-Witten invariants using an argument well-documented in the literature \cite{[AkhmedovBaykurPark], [BaldridgeKirk], [FintushelParkStern], [FintushelStern2], [FintushelStern]}. We reproduce the argument here for the sake of completeness, which requires us to describe the relation between the Seiberg-Witten invariants of the 4-manifolds $X_{T, S^1_b}(p/n)$, $X$ and $X_{T, S^1_b}(0/1)$. Given a characteristic element $k_0\in H_2(X_{T, S^1_b}(0/1); \Z)$, there are unique characteristic elements $k_X\in H_2(X_{T, S^1_b}(1/0); \Z) = H_2(X; \Z)$ and $k_{p/n}\in H_2(X_{T, S^1_b}(p/n); \Z)$ \cite[Remark 4]{[AkhmedovBaykurPark]}, \cite[p. 64]{[FintushelStern2]}. The 4-manifolds $X_{T, S^1_b}(1/0) = X$ and $X_{T, S^1_b}(0/1)$ will have at most one basic class up to sign in our setting; cf. \cite{[AkhmedovBaykurPark], [FintushelStern], [FintushelStern2]}. As described in \cite[Section 3]{[FintushelStern]}, the 4-manifold $X_{T, S^1_b}(0/1)$ has infinite cyclic fundamental group and it admits a symplectic structure \cite[Section 4]{[FintushelStern]} (cf. \cite[Section 3]{[FintushelParkStern]}). A result of Taubes \cite{[Taubes]} says that the canonical class $k_0 = - c_1(X_{T, S^1_b}(0/1))$ is a basic class of $X_{T, S^1_b}(0/1)$ and $\Sw_{X_{T, S^1_b}(0/1)}(\pm k_0) = \pm 1$. Moreover, the adjunction inequality (see \cite[Section 2.1]{[AkhmedovBaykurPark]}) implies that $k_0\in \mathcal{B}$ is the only basic class up to sign. 

It follows that there is a unique $k_{p/n}\in \mathcal{B}_{X_{T, S^1_b}(p/n)}$ up to sign for every $n\geq 1$, and the product formula of Morgan-Mrowka-Szabo \cite[Theorem 1.1]{[MorganMrowkaSzabo]} yields\begin{equation}\label{Formula}\Sw_{X_{T, S^1_b}(p/n)}(k_{p/n}) = p\cdot \Sw_X(k_X) + n\cdot \sum_{i} \Sw_{X_{T, S^1_b}(0/1)}(k_0 + i[T_0]).\end{equation}

There is a 2-torus $T_d\subset X_{T, S^1_b}(0/1)$ that is geometrically dual to the core 2-torus $T_0\subset X_{T, S^1_b}(0/1)$ of the surgery. Along with this fact, an adjunction inequality argument implies that the sum on the right handside of (\ref{Formula}) contains at most one non-vanishing term; see \cite[Section 4.1]{[AkhmedovBaykurPark]} for the argument. We have the equality\begin{equation} \Sw_{X_{T, S^1_b}(p/n)}(k_{p/n}) = p\cdot \Sw_X(k_X) + n\cdot \Sw_{X_{T, S^1_b}(0/1)}(k_0)\end{equation}and we conclude that there is an infinite set of pairwise non-diffeomorphic closed 4-manifolds (\ref{Infinite Set}). This concludes the proof of Theorem \ref{Theorem Infinite Smooth Structures}. 
\end{proof}

What is obtained when we set $p = 0$ in the statement of Theorem \ref{Theorem Infinite Smooth Structures} and the previous proof, is an infinite set $\{X_{T, S^1_b}(0/n): n\in \Z - \{0\}\}$ of pairwise non-diffeomorphic closed 4-manifolds with infinite cyclic fundamental group and the same homology of the connected sum $2\mathbb{CP}^2\#4\overline{\mathbb{CP}^2}\#S^1\times S^3$; cf. Fintushel-Stern \cite[Theorem 1.1]{[FintushelStern]}. 

\begin{theorem}\label{Theorem Infinite Smooth Structures Infinite}There is a smoothly embedded nullhomologous 2-torus $T\subset \mathbb{CP}^2\#3\overline{\mathbb{CP}^2}$  and a nullhomologous curve in its complement $S^1_b\subset \mathbb{CP}^2\#3\overline{\mathbb{CP}^2}\setminus \nu(T)$ such that performing a $(0/n)$-torus surgery to $\mathbb{CP}^2\#3\overline{\mathbb{CP}^2}$ on $T$ along $S^1_b$ yields an infinite set\begin{equation}\label{Infinite Set Infinite}\{X_{T, S^1_b}(0/n): n\in \Z-\{0\}\}\end{equation}of pairwise non-diffeomorphic 4-manifolds with infinite cyclic fundamental group and such that every element has the homology of the connected sum\begin{equation}\label{Topological Prototype Infinite}2\mathbb{CP}^2\#4\overline{\mathbb{CP}^2}\# S^1\times S^3.\end{equation}
\end{theorem}

Similar statements to Theorem \ref{Theorem Infinite Smooth Structures} and Theorem \ref{Theorem Infinite Smooth Structures Infinite} for further choices of homeomorphism types of 4-manifolds with cyclic fundamental group are produced by employing other results of Fintushel-Stern in \cite{[FintushelStern]}.

\subsection{2-spheres whose complement has a prescribed fundamental group}\label{Section Build Spheres} Let $X$ be a closed smooth 4-manifold whose fundamental group has a presentation\begin{equation}\pi_1X = \langle g_1, \ldots, g_j: r_1, \ldots, r_k\rangle\end{equation} such that adding the relation $g_1 = 1$ to it for a given generator $g_1$, one obtains the trivial group. Cyclic groups and the group $\langle g_1, g_2: {g_1}^5 = (g_1g_2)^2 = {g_2}^3\rangle$ are examples of such groups. 

Let $\alpha_{1}\subset X$ be a based loop whose homotopy class is $[\alpha_1] = g_1\in \pi_1X$. Build the closed smooth simply connected $4$-manifold \begin{equation}\label{Ambient General}Z:= X\setminus \nu(\alpha_1) \cup (D^2\times S^2)\end{equation}and consider the belt 2-sphere\begin{equation}\label{Belt Sphere General}S:= \{0\}\times S^2\subset D^2\times S^2\subset Z.\end{equation}Further information is needed on the framing of the loop $\alpha_1\subset X$ to pin down the diffeomorphism or homeomorphism type of $Z$. Once this is taken care of, this process provides a 2-sphere (\ref{Belt Sphere General}) smoothly embedded in $Z$ and whose complement has fundamental group $G$. A topological construction of locally flat 2-surfaces in topological $4$-manifolds is obtained by using locally flat embedded submanifolds in the surgery (\ref{Ambient General}); see \cite[Section 9.3]{[FreedmanQuinn]} for existence and uniqueness results on tubular neighborhoods of locally flat embedded submanifolds. 

We set some notation and construct the 2-spheres of Theorem \ref{Theorem A} using this procedure in the following example. It includes the choice of framing on the loop whose homotopy class generates the fundamental group. 

\begin{exa}\label{Example Theorem A}Fix $p\geq 2$ and an integer $n\in \Z$. Consider the 4-manifold $X_{T, S^1_b}(p/n)$ in the set (\ref{Infinite Set}) and let $\widehat{T}\subset X_{T, S^1_b}(p/n)$ be the core 2-torus of the surgery. Let $\alpha\subset X_{T, S^1_b}(p/n)$ be a loop such that the 4-manifold\begin{equation}\label{Ambient}Z_{n, p}:= (X_{T, S^1_b}(p/n)\setminus \nu(\alpha)) \cup (D^2\times S^2)\end{equation}is simply connected and consider the belt 2-sphere\begin{equation}\label{Belt Sphere}S_{n, p}:= \{0\}\times S^2\subset D^2\times S^2\subset Z_{n, p}.\end{equation}Notice that the loop $\alpha$ lies on the boundary of $\partial\nu(\widehat{T})$. The framing on the loop $\alpha$ is induced by the product framing of core torus of the $(p/n)$-torus surgery. The complement of the 2-sphere (\ref{Belt Sphere}) has fundamental group\begin{equation}\pi_1(Z_{n, p}\setminus \nu(S_n)) = \Z/p,\end{equation} and homology class of (\ref{Belt Sphere}) satisfies $[S_{n, p}]\neq 0\in H_2(Z_{n, p}; \Z)$. Moreover, the 4-manifold $X_{n, p}$ is recovered by applying surgery to $Z_{n, p}$ along $S_{n, p}$.

\end{exa}

\subsection{The ambient 4-manifold of Theorem \ref{Theorem A} and Theorem \ref{Theorem E}}\label{Section Diffeo Type} We prove in this section that the 2-spheres (\ref{Belt Sphere}) of Example \ref{Example Theorem A} are all smoothly embedded in $2\mathbb{CP}^2\#4\overline{\mathbb{CP}^2}$, and postpone to Section \ref{Section Proof of Theorem A} the proof that they are pairwise smoothly inequivalent.

\begin{proposition}\label{Proposition Diffeo Ambient}The 4-manifold $Z_{n, p}$ defined in (\ref{Ambient}) is diffeomorphic to the connected sum $2\mathbb{CP}^2\#4\overline{\mathbb{CP}^2}$ for every $n\in \Z$ and a fixed $p\geq 2$. In particular, there is an infinite set\begin{equation}\{S_{n, p}: n\in \Z\}\end{equation} of 2-spheres smoothly embedded in $Z = 2\mathbb{CP}^2\# 4\overline{\mathbb{CP}^2}$ such that the complement $Z\setminus \nu(S_{n, p})$  has fundamental group $\Z/p$ for every $n\in \Z$. 
\end{proposition}

\begin{proof} We use an argument due to Moishezon \cite[Lemma 13]{[Moishezon]} (see Gompf's description too \cite[Lemma 3]{[Gompf]}) and work of Baykur-Sunukjian \cite{[BaykurSunukjian]} to establish the diffeomorphism type of our 4-manifolds. We follow the notation in \cite[Lemma 3]{[Gompf]}, fix an $n\in \Z$ and a $p\geq 2$, and consider the 4-manifold $X_{T, S^1_b}(p/n)$ in (\ref{Infinite Set}) that is constructed from $\mathbb{CP}^2\#3\overline{\mathbb{CP}^2}$ using torus surgeries and the 4-manifold $Z_{n, p}$ built in (\ref{Ambient}). Perform a torus surgery to $X_{T, b}(p/n)$ which identifies the loop that generates its fundamental group with the normal disk to the 2-torus to obtain a simply connected 4-manifold $\hat{N}$; this gluing map is described on \cite[page 101]{[Gompf]}. The latter 4-manifold can also be obtained by applying a torus surgery to $\mathbb{CP}^2\#3\overline{\mathbb{CP}^2}$. Moishezon's argument implies that $Z_{n, p}$ is obtained from $\hat{N}$ by doing surgery along a loop \cite[\S 5.2]{[GompfStipsicz]}, i.e., $Z_{n, p} = N^\ast =  \hat{N}\#S^2\times S^2$  \cite[Proposition 5.2.3, Proposition 5.2.4]{[GompfStipsicz]}. Results of Baykur-Sunukjian \cite[Section 3]{[BaykurSunukjian]} imply that $\hat{N}\#S^2\times S^2$ is diffeomorphic to $\mathbb{CP}^2\#3\overline{\mathbb{CP}^2}\#S^2\times S^2 = 2\mathbb{CP}^2\#4\overline{\mathbb{CP}^2}$. Since the choice of $n$ and $p$ was arbitrary, we conclude that $Z_{n, p}$ is diffeomorphic to $2\mathbb{CP}^2\#4\overline{\mathbb{CP}^2}$ for every $n\in \Z$ and $p\geq 2$.
\end{proof}

A tweak to the proof of Proposition \ref{Proposition Diffeo Ambient} pins down the diffeomorphism type of the 4-maniflds constructed in the proof of Theorem \ref{Theorem E}. 

\subsection{Topological isotopy}\label{Section TOPIsotopy}Locally-flat embeddings of 2-spheres in 4-manifolds whose complement has finite cyclic fundamental group have been studied by Lee-Wilczyn\'ski \cite{[LeeWilczynski]} and Hambleton-Kreck \cite[Theorem 4.5]{[HambletonKreck]}. The next result from their work is of particular importance for our purposes.

\begin{theorem}\label{Theorem LWHK}(Lee-Wilczy\'nski, Hambleton-Kreck). Let $X$ be a closed simply connected topological 4-manifold such that $b_2(X) > |\sigma(X)| + 2$ and let $h\in H_2(X; \Z)$ be a homology class of non-zero divisibility $p\neq 0$. Let $S_1, S_2\subset X$ be locally flat embedded 2-spheres with homology classes $[S_1] = [S_2] = h\in H_2(X; \Z)$, and whose complement has fundamental group $\pi_1(X\setminus \nu(S_1)) = \Z/p = \pi_1(X\setminus \nu(S_2))$ for $p\geq 2$. If\begin{equation}b_2(X) > \max_{0\leq j <p} |\sigma(X) - 2j(p - j)(1/p^2)h\cdot h|,\end{equation}then there is a topological isotopy between $S_1$ and $S_2$.
\end{theorem}

Notice that our ambient 4-manifold $2\mathbb{CP}^2\#4\overline{\mathbb{CP}^2}$ is within the range of the hypothesis of Theorem \ref{Theorem LWHK}. Moreover, the homology class of the belt 2-sphere (\ref{Belt Sphere}) of Example \ref{Example Theorem A} has non-zero divisibility and self-intersection zero by construction. We conclude that the 2-spheres that were constructed in the previous sections are all topologically isotopic to each other by Theorem \ref{Theorem LWHK}.

\begin{corollary}\label{Corollary Topologically Isotopic}The infinite set $\{S_{n, p}: n\in \Z\}$ of Proposition \ref{Proposition Diffeo Ambient} is made of smoothly embedded 2-spheres in $2\mathbb{CP}^2\#4\overline{\mathbb{CP}^2}$ that are pairwise topologically isotopic.

\end{corollary}

\subsection{Some examples of $\mathbb{F}$-homology $4$-spheres}\label{Section Homology 4-Spheres}Constructions of 4-manifolds that have the same $\mathbb{F}$-homology as $S^4$ are not scarce in the literature. For example,  a surgery theory construction of $\Q$-homology $4$-spheres with finite cyclic fundamental group is given by Hambleton-Kreck in \cite[Proposition 4.1]{[HambletonKreck2]}. Their examples include $4$-manifolds with non-vanishing Kirby-Siebermann invariant and they admit no smooth structure. We describe two constructions of such objects in this section. 

The first involves doing surgery on the product of a 3-manifold with a circle. Spun closed smooth 4-manifolds form a classical set of examples of 4-manifolds that share the homology of $S^4$ with $\mathbb{F}$-coefficients and whose fundamental group is a 3-manifold group. We briefly recall their construction and suppose that $N$ is a closed orientable 3-manifold. A homology $4$-sphere $\Sigma_N$ with fundamental group $\pi_1\Sigma_N = \pi_1N$ is constructed as\begin{equation}\label{Homology Sphere}\Sigma_N:= (N\times S^1)\setminus \nu(\{pt\}\times S^1) \cup_{\id} (D^2\times S^2),\end{equation}where we use the identity map to identify the common boundary. There is another choice of framing, yet results of Plotnick \cite{[Plotnick]} state that there is a unique diffeomorphism class of (\ref{Homology Sphere}) for the 3-manifolds employed in this paper.
If $N$ is an $\mathbb{F}$-homology 3-sphere, then $\Sigma_N$ is an $\mathbb{F}$-homology 4-sphere. There are two principal choices of 3-manifold used in the proofs of our results. 
\begin{itemize}
\item For $N = L(p, 1)$, we obtain a $\Q$-homology 4-sphere $\Sigma_{L(p, 1)}$ with fundamental group $\pi_1\Sigma_{L(p, 1)} = \Z/p$.
\item For $N = \Sigma(2, 3, 5)$, we obtain a $\Z$-homology $4$-sphere $\Sigma_{\Sigma(2, 3, 5)}$ with fundamental group $\pi_1\Sigma_{\Sigma(2, 3, 5)} = \langle a, b : a^5 = (ab)^2 = b^3\rangle$.
\end{itemize}

A second construction of smooth $\Q$-homology 4-spheres with finite cyclic fundamental group is through handlebodies. Gompf-Stipsicz's depiction of a pair of orientable $S^2$-bundles over $\R P^2$ in \cite[Figure 5.46]{[GompfStipsicz]} describes a handlebody of a pair of $\Q$-homology 4-spheres with fundamental group of order two whose second Stiefel-Whitney class can be chosen to vanish or not depending on the $n$-framing of one of the two 2-handles. Handlebodies of pairs of $\Q$-homology 4-spheres $\{\Sigma_{p, n}: n \in \{0, 1\}\}$ with fundamental group\begin{center}$\pi_1\Sigma_{p, n} = \Z/p$\end{center}for every $p\geq 2$ and second Stiefel-Whitney class\begin{center}$w_2\Sigma_{p, n} = n$\end{center} consisting of one 0-handle, one 1-handle, one 0-framed 2-handle, one $n$-framed 2-handle, one 3-handle, and one 4-handle are drawn as a straight-forward extension of the $p = 2$ case \cite[Figure 5.46]{[GompfStipsicz]}.

\subsection{2-spheres in simply connected 4-manifolds via $\mathbb{F}$-homology 4-spheres}\label{Section Q Homology Spheres} The 4-manifolds of the previous section and the procedure of Section \ref{Section Build Spheres} yields knotted 2-spheres smoothly embedded in the total space of an $S^2$-bundle over $S^2$. The case of most interest for us is summarized in the following lemma.

\begin{lemma}\label{Lemma Spheres} (Sato \cite[\S 3]{[Sato]}). There is a smoothly embedded 2-sphere $S_p\hookrightarrow S^2\times S^2$ whose complement has fundamental group $\Z/p$ for every $p\geq 2$.

There is a smoothly embedded 2-sphere $S_G\hookrightarrow S^2\times S^2$ whose complement has fundamental group $G = \langle a, b : a^5 = (ab)^2 = b^3\rangle$ or $\Z/p$.
\end{lemma}

A variation of the proof of Proposition \ref{Proposition Diffeo Ambient} yields a proof of Lemma \ref{Lemma Spheres} by using Moishezon's argument \cite{[Moishezon]}, a lemma of Gompf \cite[Lemma 1.6]{[Gompf]} and a result of Akbulut \cite[Theorem]{[Akbulut0]}; cf. Tange \cite{[Tange]}. Another proof of Lemma \ref{Lemma Spheres} is obtained by using handlebodies \cite{[Akbulut0], [Akbulut1], [GompfStipsicz]}.

\subsection{Proof of Theorem \ref{Theorem A}}\label{Section Proof of Theorem A} We collect the results of previous sections into a proof of the following theorem, which is equivalent to Theorem \ref{Theorem A}.

\begin{theorem}\label{Theorem 1}Fix $p\geq 2$. There is an infinite set\begin{equation}\label{Infinite 2-knots}\{S_{n, p}\subset 2\mathbb{CP}^2\# 4\overline{\mathbb{CP}^2}: n\in \Z\}\end{equation}made of topologically isotopic 2-spheres whose complement has fundamental group $\Z/p$, and for which doing surgery on each element yields the infinite set (\ref{Infinite Set}) of pairwise non-diffeomorphic smooth 4-manifolds in the homeomorphism class of $\mathbb{CP}^2\#3\overline{\mathbb{CP}^2}\# \Sigma_p$. 

In particular, there is a diffeomorphism of pairs\begin{equation}(2\mathbb{CP}^2\# 4\overline{\mathbb{CP}^2}, S_{n_1, p})\rightarrow (2\mathbb{CP}^2\# 4\overline{\mathbb{CP}^2}, S_{n_2, p})\end{equation} if and only if $n_1 = n_2$, and the infinite set (\ref{Infinite 2-knots}) consists of pairwise smoothly inequivalent  2-spheres.
\end{theorem}

\begin{proof} The infinite set (\ref{Infinite 2-knots}) was constructed in Section \ref{Section Build Spheres}. The fundamental group of the complement of any 2-sphere is a prescribed finite cyclic group; see Example \ref{Example Theorem A}. Corollary \ref{Corollary Topologically Isotopic} says that elements in (\ref{Infinite 2-knots}) are pairwise topologically isotopic. As indicated in Example \ref{Example Theorem A}, the 4-manifold $X_{T, S^1_b}(p/n)$ in the infinite set (\ref{Infinite Set}) is obtained by carving out a tubular neighborhood $\nu(S_{n, p})$ of a 2-sphere in (\ref{Infinite 2-knots}) from $2\mathbb{CP}^2\#4\overline{\mathbb{CP}^2}$, and capping off the boundary with $S^1\times D^3$.  Given that the infinite set (\ref{Infinite Set}) is made of pairwise non-diffeomorphic 4-manifolds, we conclude that the infinite set (\ref{Infinite 2-knots}) is made of pairwise smoothly inequivalent 2-spheres.  
\end{proof}

\begin{remark}\label{Remark Proof}A minor modification to the previous argument yields a proof of Theorem \ref{Theorem E}. 

\end{remark}

\subsection{Proof of Theorem \ref{Theorem B}}Let $\{Z_n : n\in \Z\}$ be an infinite set of pairwise non-diffeomorphic 4-manifolds in the homeomorphism class of $Z$. Taking a connected sum with any $\Q$-homology 4-sphere $M$ yields an infinite set\begin{equation}\label{Infinite Set Manifolds}\{Z_n\# M: n\in \Z\}\end{equation}of reducible pairwise non-diffeomorphic 4-manifolds that are pairwise homeomorphic to $Z\# M$. The smooth structures are distinguished with the Seiberg-Witten invariant of the connected sums using the fact that $b_1(M) = 0 = b_2^+(M)$ and results of Kotschick-Morgan-Taubes \cite{[KotschickMorganTaubes]}. By hypothesis, there is a $\Sp^{\C}$-structure on $Z_n$ for which the Seiberg-Witten invariant $\Sw_{Z_n}$ is non-zero. As explained in \cite[Proof of Proposition 2]{[KotschickMorganTaubes]}, the $\Sp^{\C}$-structure can be extended to the connected sum $Z_n\# M$ and conclude that there is a $\Sp^{\C}$-structure for which $\Sw_{Z_n\#M} = \Sw_{Z_n}$. This implies that the infinite set $\{Z_n\#M: n\in \Z\}$ consists of pairwise non-diffeomorphic 4-manifolds that are pairwise homeomorphic to $Z\#M$.

We do surgery along the loop $\alpha\subset Z_n\# M$ as in the hypothesis of Theorem \ref{Theorem B} verbatim to the procedure described in Example \ref{Example Theorem A} to construct an infinite set\begin{equation}\label{Infinite Set Spheres Proof}\{S_{n, \pi}: n\in \Z, \pi = \pi_1M\}\end{equation} that are smoothly embedded in $Z\#S^2\times S^2$ and whose complement has fundamental group $\pi = \pi_1M$. By construction we obtain a homeomorphism of pairs between $(Z\#S^2\times S^2, S_{n_1, \pi})$ and $(Z\#S^2\times S^2, S_{n_2, \pi})$ for every $n_i\in \Z$. Surgery on the belt 2-sphere $S_{n, \pi}\subset Z\#S^2\times S^2$ gives us $Z_n\# M$ back. Since the infinite set (\ref{Infinite Set Manifolds}) is made of pairwise non-diffeomorphic 4-manifolds, we conclude that there is no diffeomorphism of pairs\begin{equation}\label{Map Pairs}(Z\#S^2\times S^2, S_{n_1, \pi})\rightarrow (Z\#S^2\times S^2, S_{n_2, \pi})\end{equation} if $n_1\neq n_2$.

\hfill $\square$

\begin{remark}\label{Remark Explanation} We elaborate on an argument to prove Theorem \ref{Theorem A} by using the construction procedure of Theorem \ref{Theorem B}. The ingredients that satisfy the hypothesis of the latter are the following. Take the infinite set $\{Z_n: n\in \Z\}$ of pairwise non-diffeomorphic 4-manifolds that are homeomorphic to $\mathbb{CP}^2\#3\overline{\mathbb{CP}^2}$ that was constructed by Fintushel-Stern in \cite{[FintushelStern]}. These 4-manifolds have different Seiberg-Witten invariant. A result of Baykur-Sunukjian \cite{[BaykurSunukjian]} implies that $Z_n\#S^2\times S^2$ is diffeomorphic to $2\mathbb{CP}^2\#4\overline{\mathbb{CP}^2}$ for every $n\in \Z$. As the 4-manifold $M$ in the statement of Theorem \ref{Theorem B}, use the $\Q$-homology 4-sphere $\Sigma_{p, 0}$ that was discussed in Section \ref{Section Q Homology Spheres} with $\pi_1\Sigma_{p, 0} = \Z/p$. Build the infinite set\begin{equation}\label{Set Elaborate}\{Z_n\#\Sigma_{p, 0}: n\in \Z\}\end{equation} of closed reducible 4-manifolds that are homeomorphic to $\mathbb{CP}^2\#3\overline{\mathbb{CP}^2}\#\Sigma_{p, 0}$. The set (\ref{Set Elaborate}) consists of pairwise non-diffeomorphic 4-manifolds, where the diffeomorphism classes are distinguished by their Seiberg-Witten invariants \cite[Proposition 2]{[KotschickMorganTaubes]}. Proceed as in the proof of Theorem \ref{Theorem B} and build an infinite set (\ref{Infinite Set Spheres Proof}) of pairwise smoothly inequivalent 2-spheres. These submanifolds have the required properties by construction and they are pairwise topologically isotopic by Theorem \ref{Theorem LWHK}.

\end{remark}

\subsection{Proof of Corollary \ref{Corollary C}}We check that the hypothesis of Theorem \ref{Theorem B} are met in these cases. As the infinite set $\{Z_n: n\in \Z\}$ we can take the infinite inequivalent smooth structures on $\mathbb{CP}^2\#3\overline{\mathbb{CP}^2}$ constructed by Fintushel-Stern \cite{[FintushelStern]}. Work of Baykur-Sunukjian \cite[Theorem ]{[BaykurSunukjian]} implies that $Z_n\#S^2\times S^2 = 2\mathbb{CP}^2\# 4\overline{\mathbb{CP}^2}$ for every $n\in \Z$; this connected sum is the simply connected 4-manifold in the statement of Corollary \ref{Corollary C}. The $\Q$-homology 4-spheres with the desired fundamental group were constructed in Section \ref{Section Homology 4-Spheres}; see Lemma \ref{Lemma Spheres}.

\hfill $\square$

\subsection{Proof of Theorem \ref{Theorem D}}Hambleton-Kreck \cite[Proposition 4.1]{[HambletonKreck2]} used surgery to prove the existence of a $\Q$-homology 4-sphere $M_p$ with non-zero second Stiefel-Whitney class $w_2(M_p)\neq 0$, non-vanishing Kirby-Siebenmann invariant $\Ks(M_p)\neq 0$, and fundamental group $\pi_1M_p = \Z/p$ for every $p\geq 2$. Carve out the loop in $M_p$ whose homotopy class generates the group $\pi_1M_p = \Z/p$, and glue back a locally flat copy of $D^2\times S^2$ to obtain a simply connected 4-manifold $\widehat{M}$ with Euler characteristic $\chi(\widehat{M}) = \chi(M_p) + 2 = 4$, signature $\sigma(\widehat{M}) = \sigma(M_p)$, second Stiefel-Whitney class $w_2(\widehat{M})\neq 0$ and Kirby-Siebenmann invariant $\Ks(\widehat{M})\neq 0$. A result of Freedman-Quinn \cite[Section 10.1]{[FreedmanQuinn]} states that $\widehat{M}$ is homeomorphic to the connected sum $\ast \mathbb{CP}^2\#\overline{\mathbb{CP}^2}$ of the Chern manifold and the complex projective space with the opposite orientation for every $p\geq 2$ (cf. \cite[Theorem 1.2.27]{[GompfStipsicz]}). The fundamental group of the complement of the belt 2-sphere $S_p$ of the surgery is isomorphic to $\pi_1M_p = \Z/p$. 

To produce the smoothly embedded 2-sphere $S_p'\subset \mathbb{CP}^2\#\overline{\mathbb{CP}^2}$ as in (\ref{Smoothly Embedded Sphere})  and prove the last clause of Theorem \ref{Theorem D}, perform surgery to the smooth $\Q$-homology 4-sphere $\Sigma_{p, 1}$ described in Section \ref{Section Q Homology Spheres}.

\hfill $\square$

\


\begin{thebibliography}{99}

\bibitem{[Akbulut0]} S. Akbulut, \emph{Scharlemann's manifold is standard}, Ann. of Math. 149 (1999), 497 - 510. 

\bibitem{[Akbulut1]} S. Akbulut, \emph{4-manifolds}, Oxford Graduate Texts in Mathematics, 25. Oxford University Press, 2016, vii + 263 pp.

\bibitem{[Akbulut2]} S. Akbulut, \emph{Isotoping 2-spheres in 4-manifolds}, in Proceedings of the 21st. G\"okova Geometry-Topology Conference 2014, 264 - 266, GGT, G\"okova, 2014.

\bibitem{[AkhmedovBaykurPark]} A. Akhmedov, R. I. Baykur and B. D. Park, \emph{Constructing infinitely many smooth structures on small 4-manifolds}, J. Topol. 1 (2008), 409 - 428.

\bibitem{[AkhmedovPark]} A. Akhmedov and B. D. Park, \emph{Exotic smooth structures on small 4-manifolds with odd signatures}, Invent. Math. 181 (2010), 577 - 603.

\bibitem{[AuckleyKimMelvinRuberman]} D. Auckly, H. J. Kim, P. Melvin and D. Ruberman, \emph{Stable isotopy in four dimensions}, Jour. London Math. Soc. 91 (2015), 439 - 463.

\bibitem{[AuckleyKimMelvinRubermanSchwartz]} D. Auckly, H. J. Kim, P. Melvin, D. Ruberman and H. Schwartz, \emph{Isotopy of surfaces in 4-manifolds after a single stabilization} Adv. in Math. 341 (2018), 609 - 615.

\bibitem{[BaldridgeKirk]} S. Baldridge and P. Kirk, \emph{Constructions of small symplectic 4-manifolds using Luttinger surgery}, J. Diff. Geom. 82 (2008), 919 - 940.

\bibitem{[BaykurSunukjian]} R. I. Baykur and N. Sunukjian, \emph{Round handles, logarithmic transforms and smooth 4-manifolds}, J. Topol. 6 (2013), 49 - 63.








\bibitem{[FintushelParkStern]} R. A. Fintushel, B. D. Park, and R. J. Stern, \emph{Reverse engineering small 4-manifolds}, Algeb. Geom. Topol. 7 (2007), 2103 - 2116.

\bibitem{[FintushelStern2]} R. A. Fintushel and R. J. Stern, \emph{Pinwheels and nullhomologous surgery on 4-manifolds with $b^+ = 1$}, Alg. Geom. Topol. 11 (2011), 1649 - 1699.

\bibitem{[FintushelStern]} R. A. Fintushel and R. J. Stern, \emph{Surgery on nullhomologous tori}, Geom. Topol. Monographs 18 (2012), 61 - 81.


\bibitem{[FreedmanQuinn]} M. H. Freedman and F. Quinn, \emph{Topology of 4-manifolds}, Princeton  Mathematical Series 39, Princeton University Press, Princeton, NJ, 1990, viii + 259 pp. 

\bibitem{[Gompf]} R. E. Gompf, \emph{Sums of elliptic surfaces}, J. Differential Geom. 34 (1991), 93 - 114.

\bibitem{[GompfStipsicz]} R. E. Gompf and A. I. Stipsicz, \emph{4-Manifolds and Kirby Calculus}, Graduate Studies in Mathematics, 20. Amer. Math. Soc., Providence, RI, 1999. xv + 557 pp.

\bibitem{[HambletonKreck]} I. Hambleton and M. Kreck, \emph{Cancellation of hyperbolic forms and topological four-manifolds}, J. Reine Angew. Math. 443 (1993), 21 - 47. 

\bibitem{[HambletonKreck2]} I. Hambleton and M. Kreck, \emph{Cancellation, elliptic surfaces and the topology of certain four-manifolds}, J. Reine Angew. Math. 444 (1993), 79 -100.




\bibitem{[KasprowskiLambertColeLandLecuona]} D. Kasprowski, P. Lambert-Cole, M. Land and A. G. Lecuona, \emph{Topologically flat embedded 2-spheres in specific simply connected 4-manifolds}, MATRIX Annals, (2019), 111 - 116. 

\bibitem{[Kim]} H. J. Kim, \emph{Modifying surfaces in 4-manifolds by twist spinning}, Geom. Topol. 10 (2006), 27 - 56.

\bibitem{[KimRuberman]} H. J. Kim and D. Ruberman, \emph{Smooth surfaces with non-simply-connected complements}, Alg. Geom. Topol. 8 (2008), 2263 - 2287.

\bibitem{[KotschickMorganTaubes]} D. Kotschick, J. W. Morgan, and C. H. Taubes, \emph{Four-manifolds without symplectic structures but with nontrivial Seiberg-Witten invariants}, Math. Res. Lett. 2 (1995), 119 - 124. 

\bibitem{[LeeWilczynski]} R. Lee and D. M. Wilczy\'nski, \emph{Locally flat 2-spheres in simply connected 4-manifolds}, Comment. Math. Helv. 6 (1990), 388 - 412.

\bibitem{[LeeWilczynski2]} R. Lee and D. M. Wilczy\'nski, \emph{Representing homology classes by locally flat surfaces of minimum genus}, Amer. J. Math. 119 (1997), 1119 - 1137. 






\bibitem{[MorganMrowkaSzabo]} J. W. Morgan, T. S. Mrowka and Z. Szab\'o, \emph{Product formulas along $T^3$ for Seiberg-Witten invariants}, Math. Res. Lett. 4 (1997), 915 - 929.

\bibitem{[Moishezon]} B. Moishezon, \emph{Complex surfaces and connected sums of complex projective planes}, Lect. Notes in Math., 603, Springer, Berlin, 1977. 


\bibitem{[Plotnick]} S. P. Plotnick, \emph{Equivariant intersection forms, knots in $S^4$, and rotations in 2-spheres}, Trans. Amer. Math. Soc. 296 (1986), 543 - 575. 






\bibitem{[Sato]} Y. Sato, \emph{Locally flat 2-knots in $S^2\times S^2$ with the same fundamental group}, Trans. Amer. Math. Soc. 323 (1991), 911 - 920.

\bibitem{[Schwartz]} H. Schwartz, \emph{Equivalent non-isotopic spheres in 4-manifolds}, J. Topol. 12 (2019), 1396 - 1412.

\bibitem{[Sunukjian]} N. S. Sunukjian, \emph{Surfaces in 4-manifolds: concordance, isotopy, and surgery}, Int. Math. Res. Not. 17 (2015), 7950 - 7978.

\bibitem{[Tange]} M. Tange, \emph{The link surgery of $S^2\times S^2$ and Scharlemann's manifolds}, Hiroshima Math. j. 44 (2014), 35 - 62. 

\bibitem{[Taubes]} C. H. Taubes, \emph{The Seiberg-Witten invariants and symplectic forms}, Math. Res. Lett. 1 (1994), 809 - 822.


\end{thebibliography}
\end{document}